\documentclass[12pt]{amsart}

\usepackage{amssymb,latexsym,amsmath,extarrows, mathrsfs, amsthm}
\usepackage{graphicx}

\usepackage[margin=1in, centering]{geometry}

\newtheorem{theorem}{Theorem}[section]
\newtheorem*{thm}{Theorem}
\newtheorem{lemma}[theorem]{Lemma}
\newtheorem{proposition}[theorem]{Proposition}
\newtheorem{remark}[theorem]{Remark}

\newtheorem*{definition}{Definition}
\newtheorem{corollary}[theorem]{Corollary}

\newtheorem*{acknowledgement}{Acknowledgement}

\newcommand{\al}{\alpha}
\newcommand{\be}{\beta}

\newcommand{\vp}{\varphi}

\newcommand{\ZZ}{\mathbb{Z}}
\newcommand{\ZC}{\mathbb{C}}

\newcommand{\ZF}{\mathbb{F}}
\newcommand{\ZFp}{\mathbb{F}_p}

\newcommand{\cA}{{\mathcal A}}

\begin{document}

\title[polynomial Roth type Theorems]{Improved estimates for polynomial Roth type theorems in Finite Fields}
\author{Dong Dong}
\address{Mathematics Department\\
University of Illinois at Urbana-Champaign\\
Urbana, IL 61801, USA}
\email{ddong3@illinois.edu}
\author{Xiaochun Li}
\address{Mathematics Department\\
University of Illinois at Urbana-Champaign\\
Urbana, IL 61801, USA}
\email{xcli@math.uiuc.edu}
\author{Will Sawin}
\address{ETH Institute for Theoretical Studies\\ETH Z\"urich\\CH-8092 Z\"urich, Switzerland}
\email{william.sawin@math.ethz.ch}

\date{\today}

\begin{abstract}
We prove that, under certain conditions on the function pair $\vp_1$ and $\vp_2$, bilinear average $p^{-1}\sum_{y\in\ZFp}f_1(x+\vp_1(y))
f_2(x+\vp_2(y))$ along curve $(\varphi_1, \varphi_2)$ satisfies certain decay estimate.
 As a consequence,
 Roth type theorems hold in the setting of finite fields. In particular, if $\vp_1,\vp_2\in\ZFp[X]$ with $\vp_1(0)=\vp_2(0)=0$ are linearly independent polynomials, then for any $A\subset \ZFp, |A|=\delta p$ with $\delta>c p^{-\frac{1}{12}}$, there are $\gtrsim \delta^3p^2$ triplets $x,x+\vp_1(y),
x+\vp_2(y)\in A$. This extends a recent result of Bourgain and Chang who initiated this type of problems, and strengthens the bound in a result of Peluse, who generalized Bourgain and Chang's work.
 The proof uses discrete Fourier analysis and algebraic geometry.
\end{abstract}

\let\thefootnote\relax
\let\thefootnote\relax

\maketitle

\section{Introduction}
\setcounter{equation}0

Fix a large prime $p$ and denote $e_p(x):=e^{2\pi i\frac{x}{p}}$. For any $\varphi_1, \varphi_2: \ZFp\to\ZFp$, we are interested in the bilinear average along the ``curve"
$\Gamma=(\varphi_1, \varphi_2)$: for any $x\in \mathbb F_p$,
\begin{equation}\label{defA}
 \cA_{\Gamma}(f_1, f_2)(x):= \frac{1}{p}\sum_{y\in\ZFp}f_1(x+\varphi_1(y))f_2(x+\varphi_2(y))\,.
\end{equation}
The behavior of the bilinear average relies closely to the following exponential sum associated to $\Gamma$,
\begin{equation} \label{defofK}
K_{\Gamma}(x,y):=
\begin{cases}
\frac{1}{p}\sum_{z\in\ZFp}e_p(x\vp_1(z)+y\vp_2(z))\quad &y\neq 0;\\
0 &y=0.
\end{cases}
\end{equation}

To state our main result, we first set up some notations. For $f:\ZFp\to\ZC$, define
\begin{equation*}
\begin{split}
&{\mathbb E}[f]={\mathbb E}_x[f]=\frac{1}{p}\sum_{x=0}^{p-1} f(x)\\
&\|f\|_r=\left(\frac{1}{p}\sum_x|f(x)|^r \right)^{\frac{1}{r}}\\
&\|f\|_{l^r}=\left(\sum_x|f(x)|^r \right)^{\frac{1}{r}}\\
&\hat{f}(z)=\frac{1}{p}\sum_x f(x)e_p(-xz)
\end{split}
\end{equation*}
With these notations, it is easy to verify that
\begin{equation*}
\begin{split}
&\|f\|_r\le\|f\|_s \text{ if } s>r \quad \text{(a special case of H\"older inequality)};\\
&\|f\|_2=\|\hat{f}\|_{l^2} \quad \text{(Parseval)}\\
&f(x)=\sum_z\hat{f}(z)e_p(xz) \quad \text{(Fourier inversion)}
\end{split}
\end{equation*}
We also need a notion of generalized diagonal sets.
\begin{definition}
A set $D\subset \ZFp\times \ZFp$ is called \textbf{generalized diagonal} if for any $x\in \ZFp$, there are $O(1)$ $y$'s such that $(x,y)\in D$ and for any $y\in\ZFp$ there are $O(1)$ $x$'s such
that $(x,y)\in D$. The implied constant must be independent of $p$.
\end{definition}

Our main theorem below provides a framework to obtain decay estimate for the bilinear operator $\cA_{\Gamma}$ associated with various function pairs $(\vp_1, \vp_2)$. Throughout the paper, $A\lesssim B$ denotes the statement that $|A|\le C|B|$ for some positive constant $C$ independent of the prime
 $p$ and the coefficients of polynomials where relevant.

\begin{theorem} \label{main thm}
Let the kernel $K_\Gamma$ be defined as in (\ref{defofK}).
We define for any $h, y, y'\in \ZF_p$,
\begin{equation}\label{defI}
I_{\Gamma}:=\sum_{x\in\ZFp}K_{\Gamma}(x,y)\overline{K_{\Gamma}(x-h,y+h)}\overline{K_{\Gamma}(x,y')}K_{\Gamma}(x-h,y'+h).
\end{equation}
Suppose that the following three conditions hold:
\begin{enumerate}
 \item There exists $\theta\in (0,1]$ such that $\frac{1}{p}\sum_{y\in\ZFp} e_p(s\vp_1(y))\lesssim p^{-\theta}$
         for any  $s\neq 0$;\label{condition: phi_1}
\item   There exists $\al\in\big(\frac{1}{4},1\big)$  such that $ K_{\Gamma}(x,y)\lesssim p^{-\al}$ for any $x,y\in {\mathbb F}_p $;   \label{condition: single K}
\item   There exists $\be>1$ such that for any $h\in \ZFp^*:=\ZF_p\backslash\{0\}$ we can find a generalized diagonal set $D_{\Gamma,h}$
   so that  $I_{\Gamma}\lesssim p^{-\beta}$ for any $(y, y')\notin D_{\Gamma, h}$.
\label{condition: 4 K}
\end{enumerate}
Then the bilinear average defined by (\ref{defA}) obeys
\begin{equation} \label{eq: conclusion of the main thm}
\|\cA_\Gamma(f_1, f_2)-\mathbb E[f_1]{\mathbb E}[f_2]\|_2\lesssim p^{-\gamma}\|f_1\|_2\|f_2\|_2
\end{equation}
with $\gamma=\min\{\theta, \al-\frac{1}{4},\frac{\be}{4}-\frac{1}{4}\}$.
\end{theorem}

Motivated by the non-conventional ergodic averages considered by Bergelson \cite{Ber} and Frantzikinakis and Kra \cite{FK}, Bourgain and Chang \cite{BC} are the first to consider quantitative estimate of the form \eqref{eq: conclusion of the main thm}. They established
(\ref{eq: conclusion of the main thm}) with $\gamma = \frac{1}{10}$ for the quadratic monomial curve $\Gamma=(y, y^2)$, via
 an elegant way combining discrete Fourier analysis, explicit evaluation of quadratic Gauss sum and
 Bombieri's estimate for Weil sum of rational functions \cite{Bom}.

Peluse \cite[Theorem 2.2]{P} generalized Bourgain-Chang's result to the polynomial curve $(\varphi_1(y),\varphi_2(y) )$ for any linearly independent polynomials $\varphi_1,\varphi_2$. Her result also applies over arbitrary finite fields of large characteristic (and not just $\mathbb F_p$). However, she must take $\gamma=1/16$. Her method is based on careful analysis of the dimension of varieties created by multiple applications of Cauchy-Schwartz, and an exponential sum bound due to Kowalski.

Our result improves the decay rate from $\frac{1}{16}$ to $\frac{1}{8}$ in Peluse's bound. This also improves the decay rate from $\frac{1}{10}$ to $\frac{1}{8}$ in the cases handled by Bourgain and Chang. Moreover, in the special case $\Gamma=(y, y^2)$, our approach does not rely on Bombieri's estimate: When $\Gamma=(y, y^2)$, $K_{\Gamma}(x,y)$ is a quadratic Gauss sum which can be evaluated explicitly. Condition \eqref{condition: 4 K} can therefore be verified by
\begin{equation*}
 |I_{\Gamma}|\leq \frac{1}{p^2}\left|\sum_x e_p\left(-\frac{x^2}{4y}\right)e_p\left(\frac{(x-h)^2}{4(y+h)}\right)
e_p\left(\frac{x^2}{4y'}\right)e_p\left(-\frac{(x-h)^2}{4(y'+h)}\right)\right|\le p^{-\frac{3}{2}} \text{ for } y\neq y',
\end{equation*}
using only quadratic Gauss sum estimate. Hence $\be=\frac{3}{2}$. It is easy to check that $\theta=\al=\frac{1}{2}$, and thus $\gamma=\frac{1}{8}$.\\

To extend to the polynomial curve $\Gamma =(y, P(y))$, the condition (\ref{condition: 4 K}) can be verified by
Deligne's fundamental work on exponential sums over finite fields \cite{De}.
When extending to the bi-polynomial case, we need to use Katz's generalisation \cite{Katz1999} of Deligne's theorem
on exponential sums over smooth affine varieties. \\

 Theorem \ref{main thm} immediately implies a quantitative Roth type theorem:
\begin{corollary} \label{main coro}
Let $\vp_1,\vp_2:\ZFp\to\ZFp$ be functions satisfying conditions \eqref{condition: phi_1}, \eqref{condition: single K} and \eqref{condition: 4 K} (with parameters $\theta$, $\al$ and $\beta$, resp.) in Theorem \ref{main thm}. Then  for any $A\subset \ZFp, |A|=\delta p$ with $\delta>c p^{-\frac{2}{3}\gamma}$, $\gamma=\min\{\theta,\al-\frac{1}{4},\frac{\be}{4}-\frac{1}{4}\}$, there are $\gtrsim \delta^3p^2$ triplets $x,x+\vp_1(y),x+\vp_2(y)\in A$.
\end{corollary}
We include its short proof (which is the same as that of Corollary 1.2 in \cite{BC}) here for the reader's convenience.
Indeed, set both $f_1$ and $f_2$ to be the indicator function of the set $A$. By Cauchy-Schwarz inequality and \eqref{eq: conclusion of the main thm},
\begin{equation} \label{eq: implication}
\sum_{x,y}f(x)f(x+\vp_1(y))f(x+\vp_2(y))\geq p^2\big({\mathbb E}[f]^3-\|f\|_2\|\cA_\Gamma(f, f)-{\mathbb E}[f]^2\|_2\big)\gtrsim p^2\delta^3,
\end{equation}
from which the corollary follows. \\

One interesting case of Theorem \ref{main thm} is the following theorem:

\begin{theorem} \label{special thm}
Let $\Gamma=(\varphi_1, \varphi_2)$ with $\vp_1,\vp_2\in\ZFp[X]$, $\vp_1(0)=\vp_2(0)=0$.
Suppose that $\varphi_1, \varphi_2$ are linearly independent. Then the average function
$\cA_\Gamma$ satisfies
\begin{equation} \label{eq in special thm}
\|\cA_\Gamma(f_1, f_2)-\mathbb E[f_1]{\mathbb E}[f_2]\|_2\lesssim  p^{-1/8}\|f_1\|_2\|f_2\|_2,
\end{equation}
with the implied constant depending only on the degrees of $\vp_1$ and $\vp_2$.
\end{theorem}

As before, we can obtain the corresponding Roth type theorem in which the lower bound $p^{-\frac{1}{12}}$ of $\delta$ is slightly better than the bound $p^{-\frac{1}{15}}$ obtained in Bourgain-Chang's paper \cite{BC}.
\begin{corollary} \label{special coro}
Let $\vp_1,\vp_2\in\ZFp[X]$, $\vp_1(0)=\vp_2(0)=0$, be linearly independent. Then for any $A\subset \ZFp, |A|=\delta p$ with $\delta>c p^{-\frac{1}{12}}$,
there are $\gtrsim \delta^3p^2$ triplets $x,x+\vp_1(y),x+\vp_2(y)\in A$.
\end{corollary}

\begin{remark}
The results of this paper can be generalized to an arbitrary finite field $\mathbb{F}_q$ with $q=p^m$. In this general setting, one should be careful that the degree of the polynomial should be coprime to $p$ in order to get the
Weil's estimate \cite{CU} (and using Deligne-Katz theory). However, as we are usually only interested in the case that $p$ is very large compared with the degrees of the relative polynomials, the coprime condition is automatically
 satisfied.
\end{remark}

\begin{remark}
Some rational functions could be included in our results. For instance, when $\vp_1(y)=y$, $\vp_2(y)=\frac{1}{y}$
(this case is also considered in \cite{BC}), we can get the same conclusion as in Theorem \ref{special thm}, using Kloosterman sum estimates (Corollary 3.3. in \cite{FKP}).
\end{remark}

\begin{remark} Theorem \ref{special thm}, in the case $\varphi_1(x)=x$, implies that the polynomial $x+ \varphi_2(y-x)$ is an almost strong asymmetric expander in the sense of Tao's paper \cite{Tao}. It is possible that this result could also be established using \cite[Theorem 3]{Tao}, but we do not pursue this. \end{remark}

We will prove Theorem \ref{main thm} in the Section \ref{section: proof}. In Section \ref{section: verify} we will verify the three conditions (\ref{condition: phi_1}),
(\ref{condition: single K}), and (\ref{condition: 4 K}) for certain polynomial pairs and henceforth prove Theorem \ref{special thm}.

\section{Proof of Theorem \ref{main thm}} \label{section: proof}
\setcounter{equation}0
We prove the main theorem in this section. We follow the spirit in the second author's work on the bilinear
Hilbert transform along curves in \cite{LAP}.
First, by using Fourier inversion for $f_1$ and $f_2$, it is clear that
$$
\cA_\Gamma(f_1, f_2)(x)
=\sum_{n_1,n_2}\hat{f_1}(n_1)\hat{f_2}(n_2)e_p((n_1+n_2)x){\mathbb E}_y[e_p(n_1\vp_1(y)+n_2\vp_2(y))].
$$
Changing variables $n_2=n, n_1=s-n$, we then split the bilinear average $ \cA_\Gamma(f_1, f_2)(x)$ into three terms:
$$
\cA_\Gamma(f_1, f_2)(x) = J_1+J_2+J_3,
$$
where
\begin{equation*}
\begin{split}
&J_1=\hat{f_1}(0)\hat{f_2}(0)={\mathbb E}[f_1]{\mathbb E}[f_2],\\
&J_2=\hat{f_2}(0)\sum_{s\neq 0}\left(\hat{f_1}(s){\mathbb E}_y[e_p(s\vp_1(y))]\right)e_p(sx),\\
&J_3=\sum_s\left(\sum_{n\neq 0}\hat{f_1}(s-n)\hat{f_2}(n){\mathbb E}_y[e_p((s-n)\vp_1(y)+n\vp_2(y))]\right)e_p(sx).
\end{split}
\end{equation*}
By the assumption \eqref{condition: phi_1}, when $s\neq 0$, we get
\begin{equation}
\mathbb E_y[e_p(s\vp_1(y))]=\frac{1}{p}\sum_ye_p(s\vp_1(y)) \lesssim \frac{1}{p^{\theta}}.
\end{equation}
Therefore, using Parseval's identity, triangle inequality, and H\"older inequality, we see that
\begin{equation*}
\begin{split}
&\|\cA_\Gamma(f_1, f_2)-\mathbb E[f_1]{\mathbb E}[f_2]\|_2\le\|\widehat{J_2}\|_{l^2}+\|\widehat{J_3}\|_{l^2} \\
&\lesssim \frac{1}{p^{\theta}}\|f_1\|_2\|f_2\|_2+\left(\sum_s\left|\sum_n\hat{f_1}(s-n)\hat{f_2}(n)K_{\Gamma}(s-n,n)\right|^2\right)^{\frac{1}{2}},
\end{split}
\end{equation*}
where $K_{\Gamma}$ is given by (\ref{defofK}).\\

Set $\gamma_0£º=\min\{\al-\frac{1}{4},\frac{\be}{4}-\frac{1}{4}\}$. Hence it remains to show
\begin{equation} \label{eq: 2 gamma}
\sum_s\left|\sum_n\hat{f_1}(s-n)\hat{f_2}(n)K_{\Gamma}(s-n,n)\right|^2\lesssim \frac{1}{p^{2\gamma_0}}\|f_1\|_2^2\|f_2\|_2^2.
\end{equation}

Next we choose to employ a $TT^*$ method (Our method and Bourgain-Chang's diverge from here). The left hand side of \eqref{eq: 2 gamma} equals
\begin{equation*}
\begin{split}
\sum_s\sum_{n_1,n_2}\hat{f_1}(s-n_1)\overline{\hat{f_1}(s-n_2)}\hat{f_2}(n_1)\overline{\hat{f_2}(n_2)}K_{\Gamma}(s-n_1,n_1)\overline{K_{\Gamma}(s-n_2,n_2)},
\end{split}
\end{equation*}
which, after changing variables $n_1=v,n_2=v+h,s=u+v$, can be rewritten as
\begin{equation} \label{eq: triple sum}
\sum_h\left(\sum_{u,v}F_h(u)G_h(v)K_{\Gamma}(u,v)\overline{K_{\Gamma}(u-h,v+h)} \right),
\end{equation}
where
\begin{equation*}
\begin{split}
&F_h(x)=\hat{f_1}(x)\overline{\hat{f_1}(x-h)};\\
&G_h(x)=\hat{f_2}(x)\overline{\hat{f_2}(x+h)}.
\end{split}
\end{equation*}
When $h=0$, using condition \eqref{condition: single K}, we see that the inner double sum in \eqref{eq: triple sum} is bounded by $$
p^{-2\al}\|F_0\|_{l^1}\|G_0\|_{l^1}=p^{-2\al}\|f_1\|_2^2\|f_2\|_2^2,
$$
which is better than $p^{-2\gamma_0}\|f_1\|_2^2\|f_2\|_2^2$ as $\al>\gamma_0$. Therefore, it remains to handle the case when
$h$ is nonzero. The tool is the following bilinear form estimate, which may be interesting on its own right (see \cite{KMS} for applications of some related bilinear forms).
\begin{proposition}\label{prop: bilinear form}
Fix $h\neq 0$. Let $\vp_1,\vp_2:\ZFp\to\ZFp$ satisfy \eqref{condition: single K} and \eqref{condition: 4 K} (with parameters $\al$ and $\beta$, resp.) in Theorem \ref{main thm}. Let $\gamma_0=\min\{\al-\frac{1}{4},\frac{\be}{4}-\frac{1}{4}\}$. Then for any $F,G:\ZFp\to\ZC$,
\begin{equation} \label{eq: bilinear estimate}
\sum_{u,v}F(u)G(v)K_{\Gamma}(u,v)\overline{K_{\Gamma}(u-h,v+h)}\lesssim  \frac{1}{p^{2\gamma_0}}\|F\|_{l^2}\|G\|_{l^2}.
\end{equation}
\end{proposition}
Once this proposition is proved, one can use \eqref{eq: bilinear estimate} and apply Cauchy-Schwarz inequality a few times to \eqref{eq: triple sum} to get the desired estimate \eqref{eq: 2 gamma}.

By duality, it is easy to see that Proposition \ref{prop: bilinear form} can be reduced to the following finite field version of H\"ormander principle (see Theorem 1.1 in \cite{H} for its continuous counterpart):
\begin{lemma} \label{lemma: H}
Fix $h\neq 0$. Let $\vp_1,\vp_2:\ZFp\to\ZFp$ satisfy \eqref{condition: single K} and \eqref{condition: 4 K} (with parameters $\al$ and $\beta$, resp.) in Theorem \ref{main thm}. Let $\gamma_0=\min\{\al-\frac{1}{4},\frac{\be}{4}-\frac{1}{4}\}$. Define an operator
\begin{equation*}
T(g)(x)=\sum_y g(y)K_{\Gamma}(x,y)\overline{K_{\Gamma}(x-h,y+h)}.
\end{equation*}
Then
\begin{equation*}
\|T(g)\|_{l^2}\lesssim \frac{1}{p^{2\gamma_0}}\|g\|_{l^2}.
\end{equation*}
\begin{proof}
We will show that
\begin{equation} \label{eq: goal of the Hormander}
\|T(g)\|_{l^2}^2\lesssim \frac{1}{p^{4\gamma_0}}\|g\|_{l^2}^2.
\end{equation}
A straightforward calculation gives
\begin{align}
\|T(g)\|_{l^2}^2&=\sum_{x,y,y'}g(y)\overline{g(y')}K_{\Gamma}(x,y)\overline{K_{\Gamma}(x-h,y+h)}\overline{K_{\Gamma}(x,y')}K_{\Gamma}(x-h,y'+h) \notag \\
&\le \sum_{(y,y')\in D_{\Gamma, h}}|g(y)||g(y')||I|+\sum_{(y,y')\notin D_{\Gamma, h}}|g(y)||g(y')||I|, \label{eq: diag and off diag}
\end{align}
where $D_{\Gamma, h}$ is the generalized diagonal set in condition \eqref{condition: 4 K} and
$$
I_{\Gamma}=\sum_x K_{\Gamma}(x,y)\overline{K_{\Gamma}(x-h,y+h)}\overline{K_{\Gamma}(x,y')}K_{\Gamma}(x-h,y'+h).
$$

We estimate the two terms in \eqref{eq: diag and off diag} by different methods. Using the definition of generalized diagonal set and
 the trivial estimate $I_{\Gamma} \lesssim  \frac{p}{p^{4\al}}$ from \eqref{condition: single K}, the first term in \eqref{eq: diag and off diag} is estimated by
\begin{equation} \label{eq: bad pairs}
\sum_{(y,y')\in D_{\Gamma, h}}|g(y)||g(y')||I_{\Gamma}| \lesssim  \sum_y|g(y)|^2\frac{p}{p^{4\al}}=\frac{1}{p^{4\al-1}}\|g\|_{l^2}^2.
\end{equation}
For the second term in \eqref{eq: diag and off diag}, we use the assumption $I_{\Gamma}\lesssim \frac{1}{p^{\be}}$ for $(y,y')\notin D_{\Gamma, h}$ and Cauchy-Schwarz inequality to get the estimate
\begin{equation} \label{eq: good pairs}
\sum_{(y,y')\notin D_{\Gamma, h}}|g(y)||g(y')||I_{\Gamma,h}| \lesssim  \frac{\sqrt{p}\sqrt{p}}{p^{\be}}\|g\|_{l^2}^2=\frac{1}{p^{\be-1}}\|g\|_{l^2}^2.
\end{equation}
Combining \eqref{eq: bad pairs} and \eqref{eq: good pairs},  we obtain
\begin{equation*}
\|T(g)\|_{l^2}^2\lesssim  \max\left\{\frac{1}{p^{4\al-1}},\frac{1}{p^{\be-1}} \right\}\|g\|_{l^2}^2=\frac{1}{p^{4\gamma_0}}\|g\|_{l^2}^2,
\end{equation*}
which is exactly what we aimed for: \eqref{eq: goal of the Hormander}.
\end{proof}
\end{lemma}

\section{Proof of Theorem \ref{special thm}} \label{section: verify}
\setcounter{equation}0
To prove Theorem \ref{special thm}, first note that we can assume without loss of generality that the two polynomials $\vp_1$ and $\vp_2$ have distinct leading terms. This is because we can rewrite \eqref{eq in special thm} in its dual form as
\begin{equation} \label{var eq in special thm}
\left| \mathbb E_{x,y}f_1(x+\vp_1(y)) f_2(x+\vp_2(y))f_3(x)-\mathbb E[f_1]{\mathbb E}[f_2] \mathbb E[f_3]\right|\lesssim  p^{-1/8}\|f_1\|_2\|f_2\|_2\|f_3\|_2,
\end{equation}
and do a change of variable $x\to x+\vp_1(y)$ on the left-hand-side of \eqref{var eq in special thm} if necessary (We are indebted to Sarah Peluse for pointing this out).

We will verify that for linearly independent polynomials $\vp_1, \vp_2\in\ZFp[X]$ with distinct leading terms, the conditions \eqref{condition: phi_1},
\eqref{condition: single K} and \eqref{condition: 4 K} in Theorem \ref{main thm} are satisfied with parameters $\theta=\frac{1}{2}$, $\al=\frac{1}{2}$ and $\be=\frac{3}{2}$, resp, and thus prove Theorem \ref{special thm}
 using Theorem \ref{main thm}.

Let $d_1$ and $d_2$ denote the degrees of $\vp_1$ and $\vp_2$, resp. Without loss of generality, we assume that $d_1\le d_2$.

Conditions \eqref{condition: phi_1} and \eqref{condition: single K} can be verified in the same way, using the well-known square-root cancellation result of
Weil \cite{W} (see also \cite{CU}). Therefore, $\theta=\al=\frac{1}{2}$. Note that the linearly independence of the two polynomials is crucial to obtain
\eqref{condition: single K}.

Now we focus on the verification of condition \eqref{condition: 4 K}. We will from now on write for simplicity that $K=K_{\Gamma}$ and $I=I_{\Gamma}$. Recall that for $y\neq 0$, $$
K(x,y)=\frac{1}{p}\sum_{z\in\ZFp}e_p(x\vp_1(z)+y\vp_2(z)).
$$
 Plug in the definition of $K$, put the sum over $x$ innermost, and we see that
\begin{equation*}
\begin{split}
I&=\sum_{x\in\ZFp}K(x,y)\overline{K(x-h,y+h)}\overline{K(x,y')}K(x-h,y'+h)\\
&=\frac{1}{p^4}\sum_x\sum_{z_1,z_2,z_3,z_4}e_p[x\vp_1(z_1)+y\vp_2(z_1)-(x-h)\vp_1(z_2)-(y+h)\vp_2(z_2)-x\vp_1(z_3)\\
&\qquad -y'\vp_2(z_3)+(x-h)\vp_1(z_4)+(y'+h)\vp_2(z_4)]\\
&=\frac{1}{p^3}\sum_{\substack{z_1,z_2,z_3,z_4\\G(z_1,z_2,z_3,z_4)=0}}e_p(F(z_1,z_2,z_3,z_4)),
\end{split}
\end{equation*}
where
\begin{equation*}
\begin{split}
&G(z_1,z_2,z_3,z_4)=\vp_1(z_1)-\vp_1(z_2)-\vp_1(z_3)+\vp_1(z_4),\\
&F(z_1,z_2,z_3,z_4)=y\vp_2(z_1)+h\vp_1(z_2)-(y+h)\vp_2(z_2)-y'\vp_2(z_3)-h\vp_1(z_4)+(y'+h)\vp_2(z_4).
\end{split}
\end{equation*}
It remains to get the estimate
\begin{equation} \label{eq: final exp sum estimate}
\sum_{\substack{z_1,z_2,z_3,z_4\\G(z_1,z_2,z_3,z_4)=0}}e_p(F(z_1,z_2,z_3,z_4))\lesssim p^{\frac{3}{2}}.
\end{equation}

We need machinery of algebraic geometry to prove \eqref{eq: final exp sum estimate}. To benefit readers who are not very familiar with algebraic geometry, we first prove \eqref{eq: final exp sum estimate} in a simpler case. We assume $\vp_1(z)=z$, and consequently $\vp_2$ has degree at least $2$ by the linearly independence assumption. In this case, the restriction $G(z_1,z_2,z_3,z_4)=0$ can be dropped once $z_4$ is replaced with $z_2+z_3-z_1$. Therefore, \eqref{eq: final exp sum estimate} is reduced to
\begin{equation} \label{eq: goal for exp sum: special case}
\sum_{z_1,z_2,z_3}e_p(F(z_1,z_2,z_3,z_2+z_3-z_1)) \lesssim  p^{\frac{3}{2}}.
\end{equation}
Such character sum is studied by Deligne in his resolution of Weil conjectures:
\begin{thm}[Theorem 8.4, \cite{De}]
Let $f\in\ZFp[X_1,\dots,X_n]$ be a polynomial of degree $d\ge 1$. Suppose that $d$ is prime to $p$, and the projective hypersurface defined by the highest degree homogeneous term $f_d$ is smooth, i.e., the gradient of $f_d$ is non-zero at any point in $\{f_d=0\}\setminus \{\mathbf{0}\}$. Then
\begin{equation*}
\sum_{z_1,\dots,z_n}e_p(f(z_1,\dots,z_n))\lesssim  p^{\frac{n}{2}}.
\end{equation*}
\end{thm}
For notational convenience, we write $d=d_2$, the degree of $\vp_2$. Let $bz^d$ denote the leading term of $\vp_2(z)$. Then the highest degree homogeneous term of $F(z_1,z_2,z_3,z_2+z_3-z_1)$ is
$$
F_d(z_1,z_2,z_3)=byz_1^d-b(y+h)z_2^d-by'z_3^d+b(y'+h)(z_2+z_3-z_1)^d.
$$
We need to verify the smoothness $\{F_d=0\}$. By straightforward calculations, $\nabla F_d=\mathbf{0}$ implies
\begin{equation*}
\begin{cases}
z_1=\left(\frac{y'+h}{y}\right)^{\frac{1}{d-1}}(z_2+z_3-z_1)\\
z_2=\left(\frac{y'+h}{y+h}\right)^{\frac{1}{d-1}}(z_2+z_3-z_1)\\
z_3=\left(\frac{y'+h}{y'}\right)^{\frac{1}{d-1}}(z_2+z_3-z_1)
\end{cases}
\end{equation*}
The above system has nonzero solutions only when
\begin{equation} \label{eq: bad case}
\left(\frac{y'+h}{y+h}\right)^{\frac{1}{d-1}}+\left(\frac{y'+h}{y'}\right)^{\frac{1}{d-1}}-\left(\frac{y'+h}{y}\right)^{\frac{1}{d-1}}=1
\end{equation}

Put those pairs $(y,y')$ satisfying \eqref{eq: bad case} as a set $D_{\Gamma, h}$, and it is not hard to check that $D_{\Gamma, h}$ is generalized diagonal. By Deligne's Theorem, \eqref{eq: goal for exp sum: special case} holds for any $(y,y')\notin D_{\Gamma, h}$. This finishes the verification of condition \eqref{condition: 4 K} with $\be=\frac{3}{2}$, assuming $\vp_1(z)=z$.

Now we turn to the general case. In \cite{Katz1980}, Katz generalizes Deligne's theorem to exponential sums over smooth affine varieties, and in \cite{Katz1999}, to singular algebraic varieties. We need the following special case of \cite[Theorem 4]{Katz1999} (The reader could skip its long proof and use it as a ``black box'' on an early reading of the paper):
\begin{theorem} \label{thm: Katz}
Let $F,G\in \ZFp[X_1,\dots,X_4]$. Assume that the degree of $F$ is indivisible by $p$, the homogeneous leading term of $G$ defines a smooth projective hypersurface, and the homogeneous leading terms of $G$ and that of $F$ together define a smooth co-dimension-$2$ variety in the projective space. Then \eqref{eq: final exp sum estimate} holds, i.e.,
\begin{equation*}
\sum_{\substack{z_1,z_2,z_3,z_4\\G(z_1,z_2,z_3,z_4)=0}}e_p(F(z_1,z_2,z_3,z_4))\lesssim p^{\frac{3}{2}}.
\end{equation*}

\end{theorem}
\begin{proof} We explain in detail how to realize this theorem as a special case of Katz's theorem. We will try to explain this derivation for mathematicians who are not experts in algebraic geometry. (However, Katz's proof requires much more advanced algebraic geometry than we can go into here).

We first restate part of Katz's theorem. Then we will explain Katz's notation and how it applies to our case.

\begin{thm} [Katz, Theorem 4 \cite{Katz1999}]
Let $N$ and $d$ be natural numbers, let $k$ be a finite field in which $d$ is invertible, let $\psi: k \to \mathbb C^\times$ be an additive character. Let $X$ be a closed subscheme of $\mathbb P^N$ of dimension $d$. Let $L$ be a section of $H^0(X,\mathcal O(1))$ and $H$ a section of $H^0(X,\mathcal O(D))$. Let $V,f,\epsilon,\delta$ be defined as in \cite[pp. 878-879]{Katz1999}. If assumptions (H1)' and (H2) of \cite[pp. 878]{Katz1999} hold, and $\epsilon\leq \delta$, then \[ \left| \sum_{x \in V(k)} \psi( f(x)) \right| \leq C \times (\#k)^{(n+1+\delta)/2}\] where $C$ is a constant depending only on $N, d$, and the number and degree of the equations defining $X$.
\end{thm}

We will choose our data so that $k = \mathbb F_p$,  $V(k) =\left\{z_1,z_2,z_3,z_4 \in \mathbb F_p \mid G(z_1,z_2,z_3,z_4)=0\right\}$, $\psi( f(x)) =e_p(F(z_1,z_2,z_3,z_4))$ for $x=(z_1,z_2,z_3,z_4) \in V(k)$, $n=3$, and $\epsilon=\delta=-1$. Furthermore $C$ will depend only on the degree of $F$ and $G$.

Examining Katz's bound, and plugging in these statements, it is clear that if we can in fact choose our data in this way, while verifying Katz's conditions, we obtain exactly our stated bound.

In what remains, we will first explain all of Katz's notation that is needed to choose $(X,L,H)$ so that
$$V(k) =  \left\{z_1,z_2,z_3,z_4 \in \mathbb F_p \mid G(z_1,z_2,z_3,z_4)=0\right\}  \text{and } \psi( f(x)) =e_p(F(z_1,z_2,z_3,z_4)),$$
and second we will verify (H1)' and (H2) and calculate $\epsilon,\delta$, explaining more of Katz's notation along the way.

For the first part, because we are interested in the $\mathbb F_p$-points $V(\mathbb F_p)$ of a scheme $V$, we will describe schemes mostly by their set of $\mathbb F_p$-points (though schemes in fact have more structure than this.) First, we take $N= 4$, so $\mathbb P^N = \mathbb P^4$ is the space whose $\mathbb F_p$-points $\mathbb P^4(\mathbb F_p)$ are the set of quintuples $(z_1,z_2,z_3,z_4,z_5) \in \mathbb F_p$, not all zero, up to multiplication by nonzero scalars.  We let $\tilde{G}$ be the homogenization of $G$, where we add additional powers of $z_5$ to all the non-leading terms of $G$ to make every term have equal degree.  Let $X$ be the vanishing set of $\tilde{G}$, so that $X(\mathbb F_p)$ is the subset of $\mathbb P^4(\mathbb F_p)$ consisting of tuples $(z_1,z_2,z_3,z_4,z_5)$ with $\tilde{G}(z_1,z_2,z_3,z_4,z_5)=0$. We must choose $L$ as an element of $H^0(X,\mathcal O(1))$, which is the space of linear functions in the variables $z_1,z_2,z_3,z_4,z_5$, and we choose $L= z_5$. Now Katz defines $V$ to be the locus in $X$ where $L$ is nonzero. Hence $V(\mathbb F_p)$ is the set of tuples $(z_1,z_2,z_3,z_4,z_5)$, with $z_5$ nonzero, up to scalar multiplication, that solve the equation $\tilde{G}(z_1,z_2,z_3,z_4,z_5)=0$. For each such tuple there exists a unique scalar multiplication that sends $z_5$ to $1$, so we can express it equally as the set of tuples $(z_1,z_2,z_3,z_4)$ with $\tilde{G}(z_1,z_2,z_3,z_4,1)=0$. By construction, $\tilde{G}(z_1,z_2,z_3,z_4,1) = G(z_1,z_2,z_3,z_4)$, so $V(\mathbb F_p) =\left\{z_1,z_2,z_3,z_4 \in \mathbb F_p \mid G(z_1,z_2,z_3,z_4)=0\right\}$, as desired.

Next, because $e_p: \mathbb F_p \to \mathbb C^\times$ is an additive character, we set $\psi= e_p$. We then need to choose $H$, a homogeneous form of degree $d$ in the variables $z_1,z_2,z_3,z_4,z_5$, so that $f(x)= F(z_1,z_2,z_3,z_4)$. Katz defines $f$ as $H/L^d$. We take $d$ to be the degree of $F$ and $H$ to be the homogenization $\tilde{F}$ of $F$, just as we did with $G$. Because we are using the bijection between $4$-tuples and $5$-tuples that sends $(z_1,z_2,z_3,z_4)$ to $(z_1,z_2,z_3,z_4,1)$, we need to check that $f(z_1,z_2,z_3,z_4,1)=F(z_1,z_2,z_3,z_4)$. This follows because

 \[f(z_1,z_2,z_3,z_4,1) = \frac{\tilde{F} (z_1,z_2,z_3,z_4,1)}{ L(z_1,z_2,z_3,z_4,1)^d }=\frac{ \tilde{F} (z_1,z_2,z_3,z_4,1)}{1^d} = F(z_1,z_2,z_3,z_4).\]

 We have therefore shown how to specialize the left side of Katz's bound to the left side of our own bound. It remains to check Katz's assumptions and also the assumptions we made in applying Katz's bound. These are as follows:

 \begin{enumerate}

 \item $d$ is invertible in $k$.

 \item Katz's assumption (H1)' holds.

 \item Katz's assumption (H2) holds.

 \item $\delta=-1$.

 \item $\epsilon=-1$.

 \item $n=3$.

 \item $C$ depends only on the degree of $F$ and $G$. \end{enumerate}

 The first condition, that $d$ is invertible in $k$, is easy to interpret, as we set $k= \mathbb F_p$ and set $d$ to equal the degree of $F$, so this is equivalent to the degree of $F$ being prime to $p$, which we have already assumed in the statement of the theorem.

 Katz's assumption (H1)' is that $X$ is Cohen-Macauley and equidimensional of dimension $n\geq 1$. Because $H$ is the hypersurface defined by a single equation $\tilde{G}=0$ in $\mathbb P^4$, a smooth variety of dimension $4$, it is automatically Cohen-Macauley of dimension $3$. This verifies assumptions (2) and (6).

 Katz defines $C$ as an explicit function of his numerical data, which consists of $N$, the number $r$ of equations needed to define $X$, the degrees of those equations, and $d$. In our case $N=4$, $r=1$, the degree of the unique equation needed to define $X$ is the degree of $G$, and $d$ is the degree of $F$. Hence $C$ is some explicit function of those degrees (assumption (7)).

 Katz defines $\epsilon$ as the dimension of the singular locus of the scheme-theoretic intersection $X \cap L$. For us $L$ is the closed subset of $\mathbb P^4$ where $z_5=0$. (Katz abuses notation slightly to use $L$ also to refer to the vanishing locus of $L$.) So $X \cap L$ is the closed subset where $z_5=0$ and $\tilde{G}=0$. Because $z_5=0$, we can ignore $z_5$ and work in $\mathbb P^3$ with coordinates $z_1,z_2,z_3,z_4$. When we do this, because all non-leading monomials of $G$ were multiplied by a positive power of $z_5$ in $\tilde{G}$, all non-leading monomials become $0$ and we are left with just the zero-locus. So $X \cap L$ is the vanishing locus of the leading term of $G$ in $\mathbb P^3$, which we assumed in the statement of the theorem is a nonsingular hypersurface, so its singular locus is empty, which by convention Katz assigns dimension $-1$, verifying $\epsilon=-1$ (assumption (5)).

  Katz defines $\delta$ as the dimension of the singular locus of the scheme-theoretic intersection $X \cap L \cap H$, and (H2) is his assumption that this has dimension $n-2$. This is the joint vanishing locus of $\tilde{G}, z_5,$ and $\tilde{F}$ in $\mathbb P^4$, which for the same reason as before is the vanishing locus of the leading terms of $F$ and $G$ in $\mathbb P^3$. Because we assumed this is a smooth subscheme of codimension $2$, it has dimension $3-2=n-2$, verifying condition (H2), and its singular locus is empty and has dimension $-1$, verifying $\delta=-1$ (assumptions (3) and (4)).
\end{proof}

Now we are ready to prove \eqref{eq: final exp sum estimate} using Theorem \ref{thm: Katz}. The first two conditions in the theorem are easy to check. To check the third condition, we handle two cases separately: $d_1<d_2$ and $d_1=d_2$.

First assume $d_1<d_2$. Let $az^{d_1}$ and $bz^{d_2}$ denote the leading term of $\vp_1$ and $\vp_2$, resp. The homogeneous leading term of $G$ and $F$ are
\begin{equation*}
\begin{split}
G_{d_1}(z_1,z_2,z_3,z_4):=az_1^{d_1}-az_2^{d_1}-az_3^{d_1}+az_4^{d_1},
\end{split}
\end{equation*}
and
\begin{equation*}
F_{d_2}(z_1,z_2,z_3,z_4):=byz_1^{d_2}-b(y+h)z_2^{d_2}-by'z_3^{d_2}+b(y'+h)z_4^{d_2},
\end{equation*}
resp.
We need to show that the Jacobian matrix
\begin{equation*}
J=
\begin{bmatrix}
\nabla G_{d_1}\\ \nabla F_{d_2}
\end{bmatrix}=
\begin{bmatrix}
d_1az_1^{d_1-1} & -d_1az_2^{d_1-1} & -d_1az_3^{d_1-1} & d_1az_4^{d_1-1}\\
d_2byz_1^{d_2-1} & -d_2b(y+h)z_2^{d_2-1} & -d_2by'z_3^{d_2-1} & d_2b(y'+h)z_4^{d_2-1}
\end{bmatrix}
\end{equation*}
has full rank at any point in $\{G_{d_1}=F_{d_2}=0\}\setminus\{\mathbf{0}\}$. When $J$ has rank less than $2$, assuming $z_1z_2z_3z_4\neq 0$, we can solve for each $z_i$ and plug in $G_{d_1}=0$ to get the equation
\begin{equation} \label{equation for d1 less than d2 case}
\left(\frac{1}{y}\right)^{\frac{d_1}{d_2-1}}-\left(\frac{1}{y+h}\right)^{\frac{d_1}{d_2-1}}-\left(\frac{1}{y'}\right)^{\frac{d_1}{d_2-1}}+\left(\frac{1}{y'+h}\right)^{\frac{d_1}{d_2-1}}=0
\end{equation}
If one or two of the four variables $z_1,z_2,z_3,z_4$ are zero, then a new equation can be obtained by deleting the corresponding term(s) in the above equation. The solutions to \eqref{equation for d1 less than d2 case} and its variants lie in a generalized diagonal set. So we can apply Theorem \ref{thm: Katz} for pairs $(y,y')$ outside this set.

Secondly consider the case $d_1=d_2=d$. The homogeneous leading term of $G$ and $F$ are
\begin{equation*}
\begin{split}
G_{d}(z_1,z_2,z_3,z_4):=az_1^{d}-az_2^{d}-az_3^{d}+az_4^{d},
\end{split}
\end{equation*}
and
\begin{equation*}
F_{d}(z_1,z_2,z_3,z_4):=byz_1^{d}-(b(y+h)-ah)z_2^{d}-by'z_3^{d}+(b(y'+h)-ah)z_4^{d},
\end{equation*}
resp. The Jacobian matrix becomes
\begin{equation*}
J=
\begin{bmatrix}
\nabla G_{d}\\ \nabla F_{d}
\end{bmatrix}=
\begin{bmatrix}
daz_1^{d-1} & -daz_2^{d-1} & -daz_3^{d-1} & daz_4^{d-1}\\
dbyz_1^{d-1} & -d(b(y+h)-ah)z_2^{d-1} & -dby'z_3^{d-1} & d(b(y'+h)-ah)z_4^{d-1}
\end{bmatrix}
\end{equation*}
When $z_1z_2z_3z_4\neq 0$, $J$ has rank $1$ only when
\begin{equation} \label{equation for d1=d2}
by=b(y+h)-ah=by'=b(y'+h)-ah.
\end{equation}
One or two terms in the above equation can be dropped if the corresponding variable is zero. Since we assume that $\vp_1$ and $\vp_2$ have distinct leading terms, $a\neq b$. It is then easy to see that the solutions to \eqref{equation for d1=d2} and its variants form a generalized diagonal set. So Theorem \ref{thm: Katz} applies in most cases, and we are done.

\bigskip
\begin{acknowledgement}
The first author would like to thank Ping Xi, Dingxin Zhang and Sarah Peluse for many helpful discussions. The third author was supported by Dr. Max R\"ossler, the Walter Haefner Foundation and  the ETH Zurich Foundation. The authors thank Ben Green and Kannan Soundararajan for bringing Peluse's paper \cite{P} into their attention.
\end{acknowledgement}


\begin{thebibliography}{20}
\bibitem{Ber} V.~Bergelson,
\emph{Ergodic Ramsey theory¡ªan update},
Ergodic theory of $\ZZ^d$ actions (Warwick, 1993¨C1994), 1-61,
London Math. Soc. Lecture Note Ser., 228, Cambridge Univ. Press, Cambridge, 1996.



\bibitem{Bom} E.~Bombieri,
\emph{On exponential sums in finite fields},
Amer. J. Math. 88 1966 71-105



\bibitem{BC}J.~Bourgain, M.C.~Chang,
\emph{Nonlinear Roth type theorems in finite fields}, Israel J. Math. (2017), https://doi.org/10.1007/s11856-017-1558-z

\bibitem{CU} L.~Carlitz, S. ~Uchiyama,
\emph{Bounds for exponential sums},
Duke Math. J. 24 (1957), 37-41.

\bibitem{De} P.~Deligne,
\emph{La conjecture de Weil. I.},
Inst. Hautes \'Etudes Sci. Publ. Math. No. 43 (1974), 273-307.



\bibitem{FK} N.~Frantzikinakis, B. ~kra,
\emph{Polynomial averages converge to the product of integrals},
Israel J. Math. 148 (2005), 267-276.

\bibitem{FKP} \'E.~Fouvry, E.~Kowalski, P.~Michel,
\emph{A study in sums of products}
Philos. Trans. Roy. Soc. A 373 (2015), no. 2040, 20140309, 26 pp.

\bibitem{H}  L.~H\"ormander,
\emph{Oscillatory integrals and multipliers on $FL^p$}
Ark. Mat. 11(1973), 1-11.

\bibitem{Katz1980} N.~Katz,
\emph{Sommes exponentielles}
Ast\'erisque, 79. Soci\'et\'e Math\'ematique de France, Paris, 1980. 209 pp.

\bibitem{Katz1999} N.~Katz,
\emph{Estimates for ``singular'' exponential sums.}
Internat. Math. Res. Notices 1999, no. 16, 875-899.

\bibitem{KMS}  E.~Kowalski, P.~Michel, W.~Sawin,
\emph{Bilinear forms with Kloosterman sums and applications}
Ann. of Math. (2) 186 (2017), no. 2, 413-500.

\bibitem{LAP} X.~Li,
\emph{Bilinear Hilbert transforms along curves I: The monomial case},
Anal. PDE 6 (2013), no. 1, 197¨C-220.


\bibitem{P} S.~Peluse
\emph{Three-term polynomial progressions in subsets of finite fields},
https://arxiv.org/abs/1707.05977

\bibitem{Tao} T.~Tao
\emph{Expanding polynomials over finite fields of large characteristic, and a regularity lemma for definable sets},
Contr. to Disc. Math. 10 (2014), no 1. 22-98.

\bibitem{W} A.~Weil,
\emph{On the Riemann hypothesis in function fields},
Proc. Nat. Acad. Sci. U. S. A. 27, (1941). 345-347.


\end{thebibliography}
\end{document}